\documentclass{amsart}
\usepackage[utf8]{inputenc}
\usepackage[english]{babel}
\usepackage{amsthm}
\usepackage{amsfonts}         
\usepackage{amsmath}
\usepackage{amssymb}
\usepackage{enumerate}

\newcommand{\R}{\mathbb{R}}
\newcommand{\C}{\mathbb{C}}

\DeclareMathOperator*{\Res}{Res}
\renewcommand\atop[2]{\genfrac{}{}{0pt}{}{#1}{#2}}

\makeatletter
\def\th@plain{%
  \thm@notefont{}
  \itshape 
}
\def\th@definition{%
  \thm@notefont{}
  \normalfont 
}
\makeatother

\theoremstyle{plain}
\newtheorem{lause}[equation]{Theorem}
\newtheorem{lem}[equation]{Lemma}

\newtheorem{kor}[equation]{Corollary}

\theoremstyle{definition}

\theoremstyle{remark}
\newtheorem{huom}[equation]{Remark}

\title[Local characterizations for the matrix monotonicity and convexity]{Local characterizations for the matrix monotonicity and convexity of fixed order}
\author{Otte Heinävaara}
\address{University of Helsinki, Department of Mathematics and Statistics, P.O. Box 68 (Gustaf Hällströmin katu 2b), FI-00014 University of Helsinki}
\email{otte.heinavaara@helsinki.fi}

\begin{document}

\begin{abstract}
	We establish local characterizations of matrix monotonicity and convexity of fixed order by giving integral representations connecting the Loewner and Kraus matrices, previously known to characterize these properties, to respective Hankel matrices. Our results are new already in the general case of matrix convexity and our approach significantly simplifies the corresponding work on matrix monotonicity. We also obtain an extension of the original characterization for matrix convexity by Kraus, and tighten the relationship between monotonicity and convexity.
\end{abstract}

\subjclass[2010]{Primary 26A48; Secondary 26A51, 47A63}
\keywords{Matrix monotone functions, Matrix convex functions}

\maketitle

\section{Introduction}

For an open interval $(a, b)$, we say that $f : (a, b) \to \R$ is matrix monotone (increasing) of order $n$ (or $n$-monotone) if for any $n \times n$ Hermitian matrices $A, B$ with spectra in $(a, b)$ and $A \leq B$ we have $f(A) \leq f(B)$.\footnote{As usual, the space of Hermitian matrices is equipped with the Loewner order, i.e. the partial order induced by the convex cone of positive semi-definite matrices.}  Analogously, $f : (a, b) \to \R$ is matrix convex of order $n$ (or $n$-convex) if for any $n \times n$ Hermitian matrices $A, B$ with spectra in $(a, b)$ and $\lambda \in [0, 1]$ we have $f(\lambda A + (1 - \lambda) B) \leq \lambda f(A) + (1 - \lambda) f(B)$.

Ever since Charles Loewner (then known as Karl Löwner) introduced matrix monotone functions in 1934 \cite{Low}, this class has been characterized in various ways. See for example \cite{Chan, Han} for survey and recent progress. The famous theorem established in the Loewner's paper states that a function that is matrix monotone of all orders on an interval, extends to upper half-plane as a Pick-Nevanlinna function: an analytic function with non-negative imaginary part. Loewner's proof of this jewel is based on an important characterization in terms of divided differences here denoted by $[\cdot, \cdot, \ldots, \cdot]_{f}$. Recall that divided differences are defined recursively by $[\lambda]_{f} = f(\lambda)$ and for distinct $\lambda_{1}, \lambda_{2}, \ldots, \lambda_{n} \in (a, b)$,
\begin{align*}
	[\lambda_{1}, \lambda_{2}, \ldots, \lambda_{n}]_{f} = \frac{[\lambda_{1}, \lambda_{2}, \ldots, \lambda_{n - 1}]_{f} - [\lambda_{2}, \lambda_{3}, \ldots, \lambda_{n}]_{f}}{\lambda_{1} - \lambda_{n}}.
\end{align*}
If $f \in C^{n - 1}(a, b)$, divided difference has continuous extension to all tuples of not necessarily distinct $n$ numbers on the interval \cite{Boo}.
\begin{lause}[Loewner]\label{basic_mon}
A function $f : (a, b) \to \R$ is $n$-monotone (for $n \geq 2$) if and only if $f \in C^{1}(a, b)$ and the Loewner matrix
\begin{align}\label{Ldef}
L = ([\lambda_{i}, \lambda_{j}]_{f})_{1 \leq i, j \leq n}
\end{align}
is positive\footnote{Here and in the following, positivity of matrix means that it is positive semi-definite.} for any tuple of numbers $(\lambda_{i})_{i = 1}^{n}$ on the same interval.
\end{lause}

Similarly Kraus, a student of Loewner introduced the matrix convexity in \cite{Kra} and established similar characterization:

\begin{lause}[Kraus]\label{basic_con}
A function $f : (a, b) \to \R$ is $n$-convex (for $n \geq 2$) if and only if $f \in C^{2}(a, b)$ and the Kraus matrix
\begin{align}\label{Krdef}
Kr = ([\lambda_{i}, \lambda_{j}, \lambda_{0}]_{f})_{1 \leq i, j \leq n}
\end{align}
is positive for any tuple of numbers $(\lambda_{i})_{i = 1}^{n} \in (a, b)^{n}$ and $\lambda_{0} \in (\lambda_{i})_{i = 1}^{n}$.
\end{lause}

A different, local characterization for monotonicity was given by another student of Loewner, Dobsch in \cite{Dob}:

\begin{lause}[Dobsch, Donoghue]\label{hankel_mon}
A $C^{2 n - 1}$ function $f : (a, b) \to \R$ is $n$-monotone if and only if the Hankel matrix
\begin{align}\label{Mdef}
M(t) = \left(\frac{f^{(i + j - 1)}(t)}{(i + j - 1)!}\right)_{1 \leq i, j \leq n}
\end{align}
is positive for any $t \in (a, b)$.
\end{lause}

By employing standard regularization techniques, one could further extend this to merely $C^{2 n - 3}$ functions with convex derivative of order $(2 n - 3)$, a class of functions for which the property makes sense for almost every $t$, to obtain the complete local characterization of the matrix monotonicity of fixed order. The result has a striking consequence: $n$-monotonicity is a local property, meaning that if function has it in two overlapping intervals, it has it for their union. This property is actually used in the proof, and although it was noted by Loewner to be easy (\cite[p. 212, Theorem 5.6]{Low}), no rigorous proof was given until 40 years later in the monograph of Donoghue \cite{Don}, and the proof is rather long when $n > 2$.

The main results of this paper establish novel integral representations connecting Hankel matrices to the Loewner and Kraus matrices. These identities give rise to a new simple proof for Theorem \ref{hankel_mon}, and more importantly, settle the conjecture in \cite{Tom} (see also \cite{Tom2}) by establishing similar local characterization for the matrix convex functions.

\begin{lause}\label{hankel_con}
A $C^{2 n}$ function $f : (a, b) \to \R$ is $n$-convex if and only if the Hankel matrix
\begin{align}\label{Kdef}
K(t) = \left(\frac{f^{(i + j)}(t)}{(i + j)!}\right)_{1 \leq i, j \leq n}
\end{align}
is positive for any $t \in (a, b)$.
\end{lause}

Again, with regularizations we may extend this to give a complete local description of matrix convexity of fixed order, which as an immediate corollary gives the expected local property theorem for convexity.

\begin{kor}\label{con_loc}
	For any positive integer $n$, $n$-convexity is a local property.
\end{kor}

As another byproduct, we obtain a slight improvement to Theorem \ref{basic_con}, where $\lambda_{0}$ may now vary freely. This also implies through divided differences a rather direct connection between matrix monotonicity and convexity.

\section{Matrix monotone functions}

\subsection{Integral representation}

In this section we construct the integral representations for the Loewner matrices alluded to in the introduction.

Let $n \geq 2$, $(a, b)$ be an interval, and $\Lambda = (\lambda_{i})_{i = 1}^{n} \in (a, b)^{n}$ be an arbitrary sequence of distinct points in $(a, b)$.

In the following the Loewner and respective Hankel matrices, introduced in the introduction in (\ref{Ldef}) and (\ref{Mdef}), for sufficiently smooth $f : (a, b) \to \R$ and $\lambda_{0} \in (a, b)$ are denoted by $L(\Lambda, f)$ and $M_{n}(t, f)$ respectively.

Recall that as one easily verifies with Cauchy's integral formula and induction, the divided differences can be written as
\begin{align}\label{cauchy_divided}
	[\lambda_{1}, \lambda_{2}, \ldots, \lambda_{n}]_{f} = \frac{1}{2 \pi i} \int_{\gamma} \frac{f(z)}{(z - \lambda_{1})\cdots (z - \lambda_{n})}dz,
\end{align}
for analytic $f$ and suitable closed curve $\gamma$.\footnote{For our purposes, it is enough to consider $f$ analytic in an open half-plane and $\gamma$ a circle in this half-plane enclosing the points $\lambda_{1}, \lambda_{2}, \ldots, \lambda_{n}$.}

Divided differences also admit a natural generalization for the mean value theorem \cite{Boo}. Namely, for an open interval $(a, b)$, $f \in C^{n - 1}(a, b)$ and any tuple of (not necessarily distinct) real numbers $\Lambda = (\lambda_{i})_{i = 1}^{n} \in (a, b)^{n}$ we have
\begin{align}\label{mean_value}
	[\lambda_{1}, \lambda_{2}, \ldots, \lambda_{n}]_{f} = \frac{f^{(n - 1)}(\xi)}{(n - 1)!}
\end{align}
for some $\xi \in [\min(\Lambda), \max(\Lambda)]$.

We shall also need the very basic properties of regularizations. Namely for even, non-negative and smooth function $\phi$ supported on $[-1, 1]$ and with integral $1$, and integrable $f : (a, b) \to \R$, regularization (or $\varepsilon$-regularization, to be precise) of $f$, denoted by $f_{\varepsilon} : (a + \varepsilon, b - \varepsilon) \to \R$ is the convolution
\begin{align*}
	f_{\varepsilon}(x) = \int_{-\infty}^{\infty} f(x - \varepsilon y) \phi(y) d y.
\end{align*}
This is a smooth function, and for any continuity point $x \in (a, b)$ of $f$ we clearly have $\lim_{\varepsilon \to 0} f_{\varepsilon}(x) = f(x)$. Note that regularizations of matrix monotone (convex) functions are obviously matrix monotone (convex) functions on a slightly smaller interval.

Define the functions $g_{j}$ for $1 \leq j \leq n$ by
\begin{align}\label{gdef}
	g_{j, \Lambda}(t, y) = \prod_{k \neq j}(1 + y (t - \lambda_{k})).
\end{align}
Define also the matrix $C(t) := C(t, \Lambda)$ by setting $C_{i, j}$ to be the coefficient of $y^{i - 1}$ in the polynomial $g_{j}(t, y)$, i.e. we have
\begin{align}\label{Cdef}
	g_{j}(t, y) = C_{1, j}(t) + C_{2, j}(t) y + \ldots + C_{n, j}(t) y^{n - 1}.
\end{align}

Define polynomial $p_{\Lambda}$ with $p_{\Lambda}(t) := \prod_{i = 1}^{n}(t - \lambda_{i})$. Also for any $z \in \C$ define function $h_{z}$ by setting $h_{z}(x) = (z-x)^{-1}$.

\begin{lem}\label{inverse_d}
	For $\Lambda = (\lambda_{i})_{i = 1}^{n}$ as before, $t \in \R$, and $z \in \C$ distinct from $t$, we have
	\begin{align*}
		C^{T}\left(t, \Lambda\right)M_{n}\left(t, h_{z} \right)C\left(t, \Lambda \right) = L(\Lambda, h_{z}) \frac{p_{\Lambda}(z)^2}{(z - t)^{2 n}}.
	\end{align*}
\end{lem}
\begin{proof}
	Write $D = C^{T}(t, \Lambda)M_{n}(t, h_{z})C(t, \Lambda)$. Note that as we have $h_{z}^{(k)}(t)/k! = (z - t)^{-k - 1}$, we may write $M_{n}(t, h_{z}) = \frac{1}{(z - t)^{2}}v v^{T}$ with $v = (1, \frac{1}{z - t}, \frac{1}{(z - t)^{2}}, \ldots, \frac{1}{(z - t)^{n - 1}})^{T}$. Thus
	\begin{align*}
		D = \frac{1}{(z - t)^{2}} (C(t,\Lambda)^{T} v) (C(t,\Lambda)^{T} v)^{T}.
	\end{align*}
	One also easily sees that $(C(t,\Lambda)^{T} v)_{i} = g_{i}(t, \frac{1}{z - t})$ so that finally
	\begin{align*}
		D_{i, j} = \frac{g_{i}(t, \frac{1}{z - t})g_{j}(t, \frac{1}{z - t}) }{(z - t)^2} = \frac{1}{(z - t)^{2}}\prod_{k \neq i}\left(1 + \frac{t - \lambda_{k}}{z - t}\right)\prod_{k \neq j}\left(1 + \frac{t - \lambda_{k}}{z - t}\right) = [\lambda_{i}, \lambda_{j}]_{h_{z}} \frac{p_{\Lambda}(z)^2}{(z - t)^{2 n}} .
	\end{align*}
\end{proof}

Consider now the function
\begin{align*}
S(z, t) := S_{\Lambda}(z, t) := -\frac{(z - t)^{2 n - 2}}{p_{\Lambda}(z)^{2}}.
\end{align*}
As $S(z, t)$ decays as $z^{-2}$, with the residue theorem we see that for suitable closed curve $\gamma$ we have
\begin{align*}
	0 = \frac{1}{2 \pi i}\int_{\gamma} S(z, t) d z = \sum_{i = 1}^{n} \Res_{z = \lambda_{i}} S(z, t).
\end{align*}
Defining now the weight functions $I_{i} := I_{i, \Lambda}$ for $1 \leq i \leq n$ by
\begin{align*}
	I_{i}(t) = \Res_{z = \lambda_{i}} S(z, t),
\end{align*}
and 
\begin{align*}
I(t) := I_{\Lambda}(t) := \sum_{\atop{1 \leq i \leq n}{\lambda_{i} < t}}I_{i}(t),
\end{align*}
we see by simple computation that $I_{i}$`s are polynomials such that $I_{i}(\lambda_{i}) = 0$ and $I$ is hence piecewise polynomial, continuous function supported on $[\min(\Lambda), \max(\Lambda)]$.

Note that with Cauchy's integral formula we can also write $I$ in the form
\begin{align*}
	I(t) = \frac{1}{2 \pi i}\int_{t - i \infty}^{t + i \infty} S(z, t) d z,
\end{align*}
whenever $t \notin \Lambda$.

\begin{huom}
	The weight function $I$ and the analogous weight $J$ to be introduced in the convex setting are examples of weights called Peano kernels or B-splines. The properties of these kernels are discussed for example in \cite{Boo2}. To stay self-contained, we give proofs of the crucial properties used in our discussion.
\end{huom}

\begin{lem}\label{crazy_lemma}
	For $\Lambda = (\lambda_{i})_{i = 1}^{n}$ as before and $z \in \C$ outside the interval $[\min(\Lambda), \max(\Lambda)]$, we have
	\begin{align*}
		(2 n - 1)\int_{-\infty}^{\infty} \frac{I(t)}{(z - t)^{2 n}} d t = \frac{1}{p_{\Lambda}(z)^2}.
	\end{align*}
\end{lem}
\begin{proof}
	We simply compute that
	\begin{eqnarray*}
		(2 n - 1)\int_{-\infty}^{\infty} \frac{I(t)}{(z - t)^{2 n}} d t &=& (2 n - 1)\sum_{i = 1}^{n} \int_{\lambda_{i}}^{\infty}\frac{I_{i}(t)}{(z - t)^{2 n}} d t \\
		&=& -(2 n - 1)\sum_{i = 1}^{n}\Res_{w = \lambda_{i}} \int_{\lambda_{i}}^{\infty} \frac{(w - t)^{2 n - 2}}{p_{\Lambda}(w)^2 (z - t)^{2 n}} d t \\
		&=& \sum_{i = 1}^{n}\Res_{w = \lambda_{i}} \frac{(1 - \frac{z - w}{z - \lambda_{i}})^{2 n - 1} - 1}{(w - z) p_{\Lambda}(w)^2} \\
		&=& - \sum_{i = 1}^{n}\Res_{w = \lambda_{i}} \frac{1}{(w - z) p_{\Lambda}(w)^2} \\
		&=& \Res_{w = z} \frac{1}{(w - z) p_{\Lambda}(w)^2} - \frac{1}{2 \pi i}\int_{\gamma} \frac{d w}{(w - z) p_{\Lambda}(w)^2} \\
		&=& \frac{1}{p_{\Lambda}(z)^2},
	\end{eqnarray*}
	where we used the residue theorem for the function $(w \mapsto (w - z)^{-1}p_{\Lambda}(w)^{-2})$.
\end{proof}

We are then ready to formulate and prove the integral representation of the Loewner matrix.

\begin{lause}\label{crazy_formula}
	For $f \in C^{2 n - 1}(a, b)$ and $\Lambda$ as before, we have
	\begin{align*}
		L(\Lambda, f) = (2 n - 1)\int_{-\infty}^{\infty} C^{T}(t, \Lambda)M_{n}(t, f)C(t, \Lambda)I_{\Lambda}(t)d t.
	\end{align*}
\end{lause}
\begin{proof}
	For entire $f$, by Lemmas \ref{inverse_d}, \ref{crazy_lemma}, Fubini and (\ref{cauchy_divided}) we have
	\begin{eqnarray*}
		(2 n - 1)\int_{-\infty}^{\infty} C^{T}(t)M_{n}(t, f)C(t)I(t)d t &=& \frac{1}{2 \pi i} \int_{\gamma} \left((2 n - 1)\int_{-\infty}^{\infty} C^{T}(t)M_{n}(t, h_{z})C(t)I(t)d t \right) f(z) d z \\
		&=& \frac{1}{2 \pi i} \int_{\gamma} L(\Lambda, h_{z}) \left((2 n - 1)\int_{-\infty}^{\infty}  \frac{p_{\Lambda}(z)^2}{(z - t)^{2 n}}I(t)d t \right) f(z) d z \\
		&=& \frac{1}{2 \pi i} \int_{\gamma} L(\Lambda, h_{z}) f(z) d z \\
		&=& L(\Lambda, f).
	\end{eqnarray*}
	The general case now follows by uniformly approximating $f$ and its derivatives up to order $(2 n - 1)$ by entire functions on $[\min(\Lambda), \max(\Lambda)]$, say, by polynomials with a suitable application of Weierstrass approximation theorem.
\end{proof}

\subsection{Positivity of the weight}

In this section we prove the non-negativity of the weight function $I$ introduced in the previous section. We begin with a simple lemma.

\begin{lem}\label{simple_lemma}
	Let $n$ be a positive integer and numbers $Z = (\zeta_{i})_{i = 1}^{n}$ non-negative. Now if $f(t) = \prod_{i = 1}^{n}(\zeta_{i} - t)^{-1}$, then for any non-negative integer $k$ and $t < 0$ we have
	\begin{align*}
		f^{(k)}(t) \geq 0.
	\end{align*}
\end{lem}
\begin{proof}
	The case of $n = 1$ is trivial; the general case follows now immediately from the product rule.
\end{proof}

\begin{lem}\label{mon_pos}
	For $\Lambda$ as before, $I_{\Lambda}$ is non-negative.
\end{lem}
\begin{proof}
	We may clearly assume that $\Lambda$ is strictly increasing. When checking the non-negativity at a point $t$, we may without loss of generality assume that $t = 0 \in [\lambda_{1}, \lambda_{n}]$. Also by continuity we may further assume that all the $\lambda_{i}$`s are non-zero. We are left to investigate
	\begin{align*}
		\frac{1}{2 \pi i}\int_{-i \infty}^{i \infty} S(z, 0) d z = -\frac{1}{2 \pi i}\int_{-i \infty}^{i \infty} \frac{z^{2 n - 2} d z}{p_{\Lambda}(z)^{2}}.
	\end{align*}
	Making the change of variable $w = \frac{1}{z}$, we are to check that
	\begin{align*}
		\frac{1}{2 \pi i}\int_{-i \infty}^{i \infty} \frac{d w}{p_{Z}(w)^{2}} \geq 0,
	\end{align*}
	where $Z = \frac{1}{\Lambda}$, that is $\zeta_{i} = \frac{1}{\lambda_{i}}$.

	Let $k$ $(< n)$ be the number of the negative $\zeta_{i}$`s and denote $Z_{-} = (\zeta_{i})_{i = 1}^{k}$. Note that if we further write $f(t) = \left(\prod_{i > k}(t - \zeta_{i})\right)^{-2}$, we have by suitable variant of (\ref{cauchy_divided})
	\begin{align*}
		\frac{1}{2 \pi i}\int_{-i \infty}^{i \infty} \frac{d w}{p_{Z}(w)^{2}} = \frac{1}{2 \pi i}\int_{-i \infty}^{i \infty}\frac{f(w) d w}{p_{Z_{-}}(w)^{2}} = [\zeta_{1}, \zeta_{1}, \zeta_{2}, \zeta_{2}, \ldots, \zeta_{k}, \zeta_{k}]_{f},
	\end{align*}
	which is positive in the view of (\ref{mean_value}) and Lemma \ref{simple_lemma}.
\end{proof}

\subsection{Characterizations for the matrix monotonicity}

\begin{proof}[Proof of Theorem \ref{hankel_mon}]
	The necessity of the condition can be found in \cite{Dob}. For sufficiency note that by Theorem \ref{crazy_formula} we can write
	\begin{align*}
		L(\Lambda) = (2 n - 1)\int_{-\infty}^{\infty} C^{T}(t)M(t)C(t)I(t)d t
	\end{align*}
	Now if $M(t) \geq 0$ for any $t \in (a, b)$, also $C^{T}(t)M(t)C(t) \geq 0$ for any $t \in (a, b)$. It follows from Lemma \ref{mon_pos} that the integrand is a positive matrix, so indeed, $L$ is positive as an integral of positive matrices. But now $f$ is $n$-monotone by Theorem \ref{basic_mon}.
\end{proof}

Putting everything together we obtain complete characterizations of the class of $n$-monotone functions.

\begin{lause}[Loewner, Dobsch, Donoghue]\label{loewner-dobsch-donoghue}
Let $n \geq 2$, and $(a, b)$ be an open interval. Now for $f : (a, b) \to \R$ the following are equivalent
\begin{enumerate}[(i)]
	\item $f$ is $n$-monotone.
	\item $f \in C^{1}(a, b)$ and the Loewner matrix $L(\Lambda, f)$ is positive for any tuple $\Lambda \in (a, b)^{n}$.
	\item $f \in C^{2 n - 3}(a, b)$, $f^{(2 n - 3)}$ is convex, and the Hankel matrix $M_{n}(t, f)$, which makes sense almost everywhere, is positive for almost every $t \in (a, b)$.
\end{enumerate}
\end{lause}
\begin{proof}
	As noted before, $(i) \Leftrightarrow (ii)$ was proven in the original paper of Loewner \cite{Low}. For $C^{2 n - 1}$ functions, $(i) \Leftrightarrow (iii)$ is Theorem \ref{hankel_mon}, and for merely $C^{2 n - 3}$ functions the claim follows from standard regularization procedure, details of which can be found in \cite{Don}. For an alternate approach to the latter equivalence, see again \cite{Don}.
\end{proof}

\begin{kor}\label{mon_loc}
	For any positive integer $n$, $n$-monotonicity is a local property.
\end{kor}

\section{Matrix convex functions}

\subsection{Integral representation}

In this section we construct the integral representations for the Kraus matrices alluded to in the introduction.

Again, let $n \geq 2$, $(a, b)$ be an interval, and $\Lambda = (\lambda_{i})_{i = 1}^{n} \in (a, b)^{n}$ be an arbitrary sequence of distinct points in $(a, b)$.

In the following the Kraus and the respective Hankel matrices, introduced in the introduction, for sufficiently smooth $f : (a, b) \to \R$ and $\lambda_{0} \in (a, b)$ are denoted by $Kr(\lambda_{0}, \Lambda, f)$ and $K_{n}(t, f)$, respectively.

The integral representation for the Kraus matrix is similar to that of the Loewner matrix. Fix again $n \geq 2$, open interval $(a, b)$ and $\Lambda = (\lambda_{i})_{i = 1}^{n} \in (a, b)^{n}$, an arbitrary sequence of distinct points on $(a, b)$. For fixed $\lambda_{0} \in (a, b)$ the weights $J_{i, \lambda_{0}} := J_{i, \lambda_{0}, \Lambda}$, now for $0 \leq i \leq n$, are defined analogously as the residues at $\lambda_{i}$'s of
\begin{align*}
T_{\lambda_{0}}(z, t) := T_{\lambda_{0}, \Lambda}(z, t) := -\frac{(z - t)^{2 n - 1}}{(z - \lambda_{0}) p_{\Lambda}(z)^{2}}
\end{align*}
and
\begin{align*}
	J_{\lambda_{0}}(t) := J_{\lambda_{0}, \Lambda}(t) := \sum_{\atop{0 \leq i \leq n}{\lambda_{i} < t}}J_{i, \lambda_{0}}(t).
\end{align*}

\begin{lem}\label{crazy_convex_lemma}
	For $\Lambda = (\lambda_{i})_{i = 1}^{n}$, as before, $\lambda_{0} \in (a, b)$ and $z \in \C$ outside the interval $[\min(\Lambda), \max(\Lambda)]$, we have
	\begin{align*}
		2 n\int_{-\infty}^{\infty} \frac{J_{\lambda_{0}}(t)}{(z - t)^{2 n + 1}} d t = \frac{1}{(z - \lambda_{0})p_{\Lambda}(z)^2}.
	\end{align*}
\end{lem}
\begin{proof}
	Proof is almost identical to that of Lemma \ref{crazy_lemma}; we just perform the residue trick with the map $(w \mapsto (w - z)^{-1}(w - \lambda_{0})^{-1}p_{\Lambda}(w)^{-2})$ instead.
\end{proof}

\begin{lause}\label{convex_crazy_formula}
	For $f \in C^{2 n}(a, b)$, $\Lambda$ as before, and $\lambda_{0} \in (a, b)$, we have
	\begin{align*}
		Kr(\lambda_{0}, \Lambda, f) = 2 n \int_{-\infty}^{\infty} C^{T}(t, \Lambda)K_{n}(t, f)C(t, \Lambda)J_{\lambda_{0}, \Lambda}(t) d t.
	\end{align*}
\end{lause}
\begin{proof}
	After noting that $K_{n}(t, h_{z}) = \frac{1}{z - t} M_{n}(t, h_{z})$, the calculation is carried out as in the proof of Theorem \ref{crazy_formula}, using Lemma \ref{crazy_convex_lemma} instead of Lemma \ref{crazy_lemma}.
\end{proof}

\subsection{Positivity of the weight}

\begin{lem}\label{con_pos}
	For $\Lambda = (\lambda_{i})_{i = 1}^{n}$ as before and $\lambda_{0} \in (a, b)$, $J_{\lambda_{0}, \Lambda}$ is non-negative.
\end{lem}
\begin{proof}
	As in the proof of Lemma \ref{mon_pos}, we can assume that $t = 0$ is our point of inspection and that $\Lambda$ is strictly increasing. We also make the same change of variables $Z = \frac{1}{\Lambda}$. Note that we may well assume that $\zeta_{0} > 0$, since the other case would follow by reflecting the variables, that is considering the sequence $-Z$ and $-\lambda_{0}$, instead. Now the inequality is reduced to an equivalent form
	\begin{align*}
		\frac{1}{2 \pi i}\int_{-i \infty}^{i \infty} \frac{d w}{(\zeta_{0} - w)p_{Z}(w)^{2}} \geq 0.
	\end{align*}
	But as in the proof of Lemma \ref{mon_pos}, the left hand side can be again written as
	\begin{align*}
		[\zeta_{1}, \zeta_{1}, \zeta_{2}, \zeta_{2}, \ldots, \zeta_{k}, \zeta_{k}]_{f}
	\end{align*}
	where $f(t) = (\zeta_{0} - t)^{-1}\left(\prod_{i > k}(t - \zeta_{i})\right)^{-2}$ and $k$ is the number of negative $\zeta_{i}$`s.
\end{proof}

\subsection{Characterizations for the matrix convexity}

\begin{proof}[Proof of Theorem \ref{hankel_con}]
	The necessity of the condition was proven in \cite{Tom}. For the other direction, by Lemma \ref{convex_crazy_formula} we can write
	\begin{align*}
		Kr(\lambda, \Lambda) = 2 n\int_{-\infty}^{\infty} C^{T}(t)K(t)C(t)J_{\lambda_{0}}(t)d t.
	\end{align*}
	But as in the proof of Theorem \ref{hankel_mon}, we see now that the Kraus matrix is an integral of positive matrices, hence positive, and Theorem \ref{basic_con} finishes the claim.
\end{proof}

The next theorem finally completes the characterization of $n$-convex functions. The original characterization of Kraus is also improved.

\begin{lause}\label{convex_complete}
Let $n \geq 2$, and $(a, b)$ be an open interval. Now for $f : (a, b) \to \R$ the following are equivalent
\begin{enumerate}[(i)]
	\item $f$ is $n$-convex.
	\item $f \in C^{2}(a, b)$ and the Kraus matrix $Kr(\lambda_{0}, \Lambda, f)$ is positive for any tuple $\Lambda \in (a, b)^{n}$ and $\lambda_{0} \in \Lambda$.
	\item $f \in C^{2}(a, b)$ and the Kraus matrix $Kr(\lambda_{0}, \Lambda, f)$ is positive for any tuple $\Lambda \in (a, b)^{n}$ and $\lambda_{0} \in (a, b)$.
	\item $f \in C^{2 n - 2}(a, b)$, $f^{(2 n - 2)}$ is convex, and the Hankel matrix $K_{n}(t, f)$, which makes sense almost everywhere, is positive for almost every $t \in (a, b)$.
\end{enumerate}
\end{lause}
\begin{proof}
	$(i) \Leftrightarrow (ii)$ was proven in \cite{Kra}. For $C^{2 n}$ functions $(i) \Leftrightarrow (iv)$ is Theorem \ref{hankel_con}; the proof of Theorem \ref{hankel_con} also gives $(iv) \Rightarrow (iii)$ in this case. For merely $C^{2 n - 2}$ functions these claims follow from regularization techniques as in the monotone case. $(iii) \Rightarrow (ii)$ is trivial.
\end{proof}

We also get an interesting corollary connecting the monotonicity to convexity, extending a result in \cite{Ben}.
\begin{kor}\label{divided_connection}
Let $n \geq 2$, and $(a, b)$ be an open interval. If $f : (a, b) \to \R$ is $n$-convex, then for any $\lambda_{0} \in (a, b)$ the function $g = (x \mapsto [x, \lambda_{0}]_{f})$ is $n$-monotone.
\end{kor}
\begin{proof}
	Simply note that $L(\Lambda, g) = Kr(\lambda_{0}, \Lambda, f)$.
\end{proof}

\begin{huom}
	The ideas introduced in the paper can be generalized to characterize more general class of functions called matrix $k$-tone functions, introduced in \cite{Franz}. A paper discussing related questions in this more general setting is in preparation.
\end{huom}

\section{Acknowledgements}

We thank the open-source mathematical software Sage \cite{Sag} for invaluable support in discovering the main identities of this paper. We are also truly grateful to O. Hirviniemi, J. Junnila and E. Saksman, and anonymous reviewers for their helpful comments on the earlier versions of the manuscript.

\bibliography{arxiv_v3}
\bibliographystyle{abbrv}

\end{document}